\documentclass{article}
\usepackage{graphicx}
\usepackage{color}
\usepackage{amsmath,amsthm}
\usepackage{amssymb}

\newtheorem{theorem}{Theorem}[section]
\newtheorem{lemma}[theorem]{Lemma}
\newtheorem{corollary}[theorem]{Corollary}
\newtheorem{proposition}[theorem]{Proposition}

\theoremstyle{definition}
\newtheorem{definition}[theorem]{Definition}

\theoremstyle{remark}
\newtheorem{remark}[theorem]{Remark}

\numberwithin{equation}{section}

\def\d{\mathrm{d}}
\def\I{\mathrm{I}}
\def\R{\mathrm{R}}

\def\Cset{\mathbb{C}}
\def\Kset{\mathbb{K}}
\def\Lset{\mathbb{L}}

\def\Pset{\mathbb{P}}
\def\Rset{\mathbb{R}}
\def\Zset{\mathbb{Z}}

\def\loc{\mathrm{loc}}
\def\mon{\mathrm{mon}}

\def\Im{\mathrm{Im}}
\def\Re{\mathrm{Re}}

\def\half{{\textstyle\frac{1}{2}}}
\def\third{{\textstyle\frac{1}{3}}}

\def\epsilon{\varepsilon}

\def\ie{{i.e.}}
\def\eg{{e.g.}}

\begin{document}

\title{Galoisian approach for a Sturm-Liouville problem on the infinite interval}
\author{David~Bl\'azquez-Sanz \& Kazuyuki Yagasaki
\thanks{This work was partially supported
 by the Japan Society for the Promotion of Science (JSPS)
 through Grant-in-Aid for JSPS Fellows No.~21$\cdot$09222.
DBS acknowledges support
 from MICINN-FEDER grant MTM2009-06973
 and CUR-DIUE grant 2009SGR859,
 and thanks his colleagues at Instituto de Matem\'aticas y sus Aplicaciones, 
 Universidad Sergio Arboleda, for helpful discussion and encouragement.
KY appreciates support from the JSPS
 through Grant-in-Aid for Scientific Research (C) Nos.~21540124 and 22540180.}}
\date{August 25, 2010}

\maketitle

\begin{abstract}
We study a Sturm-Liouville type eigenvalue problem
 for second-order differential equations 
 on the infinite interval $(-\infty,\infty)$.
Here the eigenfunctions are nonzero solutions exponentially decaying at infinity.
We prove that at any discrete eigenvalue
 the differential equations are integrable in the setting of differential Galois theory
 under general assumptions.
Our result is illustrated with two examples
 for a stationary Schr\"odinger equation having a generalized Hulth\'en potential
 and an eigenvalue problem for a traveling front in the Allen-Cahn equation.
\end{abstract}


\section{Introduction}

We study a Sturm-Liouville type problem
 for second-order differential equations of the form
\begin{equation}\label{SLTE}
\frac{\d^2\psi}{\d x^2}+\mu(x)\frac{\d\psi}{\d x}+\nu(x)\psi=\lambda\psi,\quad
\psi,\lambda\in\Cset,
\end{equation}
on the infinite interval $(-\infty,\infty)$ with boundary conditions
\begin{equation}\label{BC}
\lim_{x\to\pm\infty}\psi(x)=0,
\end{equation}
where $\mu,\nu:\Rset\to\Rset$ are analytic functions.
If the boundary value problem (\ref{SLTE},\ref{BC}) has a nonzero solution,
 then the value of $\lambda$ is called an \emph{eigenvalue}
 and the nonzero solution $\psi(x)$,
 which is easily shown to decay exponentially at infinity,
 is called the associated \emph{eigenfunction}.
It is also a well-known fact that a classical Sturm-Liouville problem
\[
-\frac{\d}{\d x}\left(p(x)\frac{\d\psi}{\d x}\right)+q(x)\psi
 =\lambda w(x)\psi,
\]
where $p,q,w:\Rset\to\Rset$ are analytic functions,
 can be casted into \eqref{SLTE} with $\mu(x)\equiv 0$
 under changes of independent and dependent variables.
See \cite{AHP05,Z05} and references therein for the history and general results
 on the Sturm-Liouville problem.
These types of equations arise in many mathematical and physical applications
 including stationary Schr\"odinger equations \cite{LL58}
 and eigenvalue problems for spectral stability of pulses and fronts
 in partial differential equations (PDEs) \cite{S02}.

In general, it is difficult to solve the eigenvalue problem
 (\ref{SLTE},\ref{BC}) analytically,
 and explicit solutions are obtained only in special cases.
For stationary Schr\"odinger equations, in which $\mu(x)\equiv 0$ in \eqref{SLTE},
 Acosta-Hum\'anez \cite{Acosta} recently studied the eigenvalue problem
 by means of differential Galois theory \cite{K57,VS03}.
Here the differential Galois theory
 is an extended version of the classical Galois theory,
 which treats the solvability of algebraic equations, for differential equations
 and deals with the problem of integrability by quadratures for them.
He computed such values of $\lambda$ as
 equation~\eqref{SLTE} with $\mu(x)\equiv 0$
 has a solvable differential Galois group for many examples
 and showed for some of them
 that the differential Galois group is solvable if $\lambda$ is an eigenvalue
 (see also \cite{Acosta2}).
In this paper, we show that this statement holds
 for \eqref{SLTE} with $\mu(x)\not\equiv 0$ under general assumptions.
More precisely, we state our main results as below.

\begin{figure}[t]
\centering\includegraphics[scale=0.8]{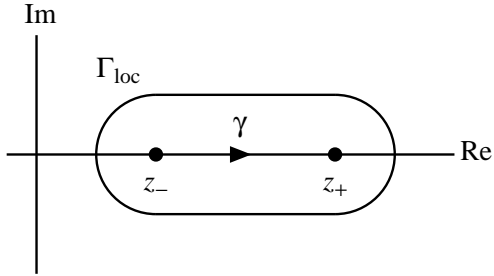}
\caption{Assumption~(A1) and a simply connected neighborhood $\Gamma_\loc$.
\label{fig:A1}}
\end{figure}

We first make the following assumptions:
\begin{enumerate}
\item[\bf(A1)]
Let $I\subset\Rset$ be an open interval.
There exist an analytic function $f:I\rightarrow\Rset$ and two points $z_\pm\in I$ 
 such that $f(z_\pm)=0$, $f'(z_\pm)\neq 0$ and
\[
\mu(x)=g(\gamma(x)),\quad
\nu(x)=h(\gamma(x)),
\]
where the prime represents differentiation with respect to $x$;
 $\gamma(x)$ is a heteroclinic solution in
\begin{equation}
\label{eqn:1d}
\frac{\d z}{\d x}=f(z),\quad
z\in\Rset.
\end{equation}
with $\lim_{x\rightarrow\pm\infty}\gamma(x)=z_\pm$;
 and $g(z),h(z)$ are meromorphic functions in an open set $U$
 containing $\{\gamma(x)\,|\,x\in\Rset\}\cup\{z_\pm\}$ in $\Cset$.
See Fig.~\ref{fig:A1}
\item[\bf(A2)]
The functions $g(z),h(z)$ are holomorphic at $z=z_\pm$.
\end{enumerate}

We easily see that $\mu(x)$ and $\nu(x)$, respectively,
 converge \emph{exponentially} to finite values
\[
\mu_\pm=g(z_\pm)
\quad\mbox{and}\quad
\nu_\pm=h(z_\pm)
\]
as $x\rightarrow\pm\infty$, since $f'(z_\pm)\neq 0$.
We also have $f'(z_+)<0$ and $f'(z_-)>0$ 
 since $z=z_+$ and $z=z_-$ must be a sink and source, respectively,
 in \eqref{eqn:1d}.

Under the transformation $z=\gamma(x)$,
 equation~\eqref{SLTE} is written as
\begin{equation}\label{CE}
\frac{\d^2\psi}{\d z^2}+\frac{g(z)+f'(z)}{f(z)}\frac{\d\psi}{\d z}
 +\frac{h(z)-\lambda}{f(z)^2}\psi=0,
\end{equation}
which is regarded as a complex differential equation
 with meromorphic coefficients and $\psi,z\in\Cset$.
Generally, a singular point in linear differential equations
 with meromorphic coefficients is called \emph{regular}
 if the growth of solutions along any ray approaching the singular point
 is bounded by a meromorphic function;
 otherwise it is called \emph{irregular}.
It is well known for second-order differential equations of the form~\eqref{CE}
 that a singular point $z_0$ is regular
 if the coefficients of $\d\psi/\d z$ and $\psi$
 are $O((z-z_0)^{-1})$ and $O((z-z_0)^{-2})$, respectively.
See, \eg, \cite{I56,WW27} for more details on this statement.
Hence, since $f'(z_\pm)\neq 0$ and $g(z),h(z)$ are holomorphic at $z_\pm$
 by assumptions~(A1) and (A2),
 we see that the singular points $z=z_\pm$ are regular in \eqref{CE}.

Let $\Gamma_\loc$ be a simply connected neighborhood
 of the path $\{\gamma(x)\,|\,x\in\Rset\}$ in $\Cset$ (see Fig.~\ref{fig:A1}).
We prove the following theorem.

\begin{theorem}
\label{thm:main}
Let $\lambda_\R=\Re(\lambda)$ and $\lambda_\I=\Im(\lambda)$.
Suppose that $z=z_\pm$ are the only singularities of \eqref{CE}
 in $\Gamma_\loc$ and
\begin{equation}
16\mu_\pm^2(\lambda_\R-\nu_\pm)+\lambda_\I^2>0.
\label{eqn:con}
\end{equation}
If the boundary value problem \emph{(\ref{SLTE},\ref{BC})} has a nonzero solution,
 then the restriction of \eqref{CE} onto $\Gamma_\loc$
 has a triangularizable differential Galois group.
\end{theorem}

Roughly speaking, this theorem means that
 if $\lambda$ is an eigenvalue satisfying \eqref{eqn:con},
 then equation~\eqref{SLTE} is integrable
 in the setting of the differential Galois theory. 
We will also see that an eigenvalue is not discrete
 if it does not satisfy \eqref{eqn:con} (see Remark~\ref{rmk:b}).

In some case
 all eigenvalues of the problem (\ref{SLTE},\ref{BC})
 have to satisfy condition~\eqref{eqn:con}.
Actually, we obtain the following result
 as a corollary of Theorem~\ref{thm:main}.

\begin{theorem}
\label{thm:main2}
Suppose that the two points $z=z_\pm$ are the only singularities of \eqref{CE}
 in $\Gamma_\loc$ and one of the following conditions holds:
\begin{enumerate}
\item[(i)] $\mu_\pm=0$;
\item[(ii)] $\mu_+=0$, $\mu_->0$;
\item[(iii)] $\mu_+<0$, $\mu_-=0$;
\item[(iv)] $\mu_+>0$, $\mu_-\ge 0$, $\nu_-\ge\nu_+$;
\item[(v)] $\mu_+\le 0$, $\mu_-<0$, $\nu_+\ge\nu_-$.
\end{enumerate}
Then the statement of Theorem~\ref{thm:main} holds.
\end{theorem}

To prove the main theorems,
 we analyze \eqref{CE} using the differential Galois theory.
Similar techniques were used
 to study bifurcations of homoclinic orbits in \cite{BY10} very recently
 and horseshoe dynamics in \cite{MP99,Y03} much earlier. 
Fauvet~\emph{et al.} \cite{FRRT10} also studied an eigenvalue problem
 for a special \emph{non-Fuchsian} second-order differential equation 
 called the prolate spheroidal wave equation \cite{WS05} on a finite interval, 
 using the differential Galois theory.
They analyzed the Stokes phenomenon
 and clarified a relation between solutions of the eigenvalue problem
 and the differential Galois group.

The rest of the paper is organized as follows. 
We provide necessary information on the differential Galois theory in Section~2
 and give proofs of Theorems~\ref{thm:main} and \ref{thm:main2} in Section~3.
In Section~4,
 our result is illustrated with two examples
 for a stationary Schr\"odinger equation
 having a generalization of the Hulth\'en potential
 \cite{H42}
 and an eigenvalue problem for a traveling front solution
 in the Allen-Cahn equation \cite{AC79}.

\section{Differential Galois theory}

We briefly review a part of the differential Galois theory
 which is often referred to as the Picard-Vessiot theory
 and gives a complete framework about the integrability by quadratures
 of linear differential equations with variable coefficients. 

\subsection{Picard-Vessiot extensions and differential Galois groups}

Consider a system of abstract differential equations
\begin{equation}\label{LinearSystem}
\partial y = Ay,\quad\quad A\in\mathrm{gl}(n,\Kset),
\end{equation}
where $\partial$ represents a \emph{derivation},
 which is an additive endomorphism satisfying the Leibniz rule;
$\Kset$ is a \emph{differential field},
 \ie, a field endowed with the derivation $\partial$;
 and $\mathrm{gl}(n,\Kset)$ denotes the ring of $n\times n$ matrices
 with entries in $\Kset$.
The set $\mathrm{C}_{\Kset}$ of elements of $\Kset$ for which $\partial$ vanishes
 is a subfield of $\Kset$
and called the \emph{field of constants for $\Kset$}.
In our application in this paper, the differential field $\Kset$ is
 the field of meromorphic functions on a Riemann surface $\Gamma$
 endowed with a meromorphic vector field,
 so that the field of constants becomes that of complex numbers, $\Cset$.
A \emph{differential field extension} $\Lset\supset \Kset$ is a field extension
 such that $\Lset$ is also a differential field
 and the derivations on $\Lset$ and $\Kset$ coincide on $\Kset$.

\begin{definition}
A \emph{Picard-Vessiot extension} for \eqref{LinearSystem}
 is a differential field extension $\Lset\supset \Kset$ satisfying the following: 
\begin{enumerate}
\item[\bf (PV1)]
There is a fundamental matrix $\Phi$
 of \eqref{LinearSystem} with coefficients in $\Lset$.
\item[\bf (PV2)]
The field $\Lset$ is spanned
 by $\Kset$ and entries of the fundamental matrix $\Phi$.
\item[\bf (PV3)]
The field of constants for $\Lset$ coincides with that for $\Kset$.
\end{enumerate}
\end{definition}

The system \eqref{LinearSystem}
 admits a Picard-Vessiot extension which is unique up to isomorphism.
If $\Kset$ is the field of meromorphic functions on a Riemann surface,
 then we have a fundamental matrix in some field of convergent Laurent series,
 and get the Picard-Vessiot extension by adding convergent Laurent series to $\Kset$.

We now fix a Picard-Vessiot extension $\Lset\supset \Kset$
and 
fundamental matrix $\Phi$ with coefficients in $\Lset$
 for \eqref{LinearSystem}.
Let $\sigma$ be a \emph{$\Kset$-automorphism} of $\Lset$,
 \ie, a field automorphism of $\Lset$
 that commutes with the derivation of $\Lset$
 and leaves $\Kset$ pointwise fixed.
Obviously, $\sigma(\Phi)$ is also a fundamental matrix of \eqref{LinearSystem}
 and consequently there is a matrix $m_\sigma$ with constant entries
 such that $\sigma(\Phi)=\Phi m_\sigma$.
This relation gives a faithful representation
 of the group of $\Kset$-automorphisms of $\Lset$
 on the general linear group as
\[
\mathrm{gal}\colon \mathrm{Aut}_\Kset(\Lset)\to \mathrm{GL}(n,\mathrm{C}_\Lset),
\quad \sigma\mapsto m_{\sigma},
\]
where $\mathrm{Aut}_{\Kset}(\Lset)$ is the set of $\Kset$-automorphisms of $\Lset$,
 and $\mathrm{GL}(n,\mathrm{C}_\Lset)$ is the group
 of $n\times n$ invertible matrices with entries in ${\rm C}_{\Lset}$.
The image of the representation ``$\mathrm{gal}$''
 is a linear algebraic subgroup of ${\rm GL}(n,{\rm C}_{\Lset})$,
 which is called the \emph{differential Galois group} of \eqref{LinearSystem}
and denoted by ${\rm Gal}(\Lset/\Kset)$.
This representation is not unique
 and depends on the choice of the fundamental matrix $\Phi$,
 but a different fundamental matrix only gives rise to a conjugated representation.
Thus, the differential Galois group is unique up to conjugation
 as an algebraic subgroup of the general linear group.

\begin{definition}
A differential field extension $\Lset\supset\Kset$ is called 
\begin{enumerate}
\item[(i)]
an \emph{integral extension} if there exists $a\in\Lset$
 such that $\dot a \in \Kset$ and $\Lset = \Kset(a)$,
 where $\Kset(a)$ is the smallest extension of $\Kset$ containing $a$;
\item[(ii)]
an \emph{exponential extension} if there exists $a\in\Lset$
 such that $\dot{a}/a \in \Kset$ and $\Lset = \Kset(a)$;
\item[(iii)]
an \emph{algebraic extension} if there exists $a\in\Lset$
 such that it is algebraic over $\Kset$ and $\Lset = \Kset(a)$.
\end{enumerate}
\end{definition}

\begin{definition}
A differential field extension $\Lset\supset\Kset$ is called a 
\emph{Liouvillian extension} if it can be decomposed as a tower of extensions,
$$
\Lset = \Kset_n \supset \ldots \supset \Kset_1\supset 
\Kset_0 = \Kset,
$$
such that each extension $\Kset_{i+1}\supset \Kset_i$ is
either integral, exponential or algebraic. It is called \emph{strictly Liouvillian}
if in the tower only integral and exponential extensions appear. 
\end{definition}

In general,
 an algebraic group $G\subset\mathrm{GL}(n,\mathrm{C}_\Lset)$
 contains a unique maximal connected algebraic subgroup $G^0$
 called the \emph{connected component of the identity}
 or \emph{connected identity component}.
The connected identity component $G^0\subset G$
 is a normal algebraic subgroup
 and the smallest subgroup of finite index,
 \ie, the quotient group $G/G^0$ is finite.
By the Lie-Kolchin Theorem \cite{K57,VS03},
 a connected solvable linear algebraic group is triangularizable.
Here a subgroup of ${\rm GL}(n,{\rm C}_{\Lset})$
 is said to be \emph{triangularizable}
 if it is conjugated to a subgroup of the group of upper triangular matrices.
The following theorem relates
 the solvability and triangularizability of the differential Galois group
 with the (strictly) Liouvillian Picard-Vessiot extension
 (see \cite{K57,VS03} and \cite{BM10} for the proofs of the first and second parts,
 respectively).

\begin{theorem}
\label{thm:dg}
Let $\Lset\supset\Kset$ be a Picard-Vessiot extension of \eqref{LinearSystem}.
\begin{enumerate}
\item[(i)]
The connected identity component
 of the differential Galois group ${\rm Gal}(\Kset/\Lset)$ is solvable
 if and only if $\Lset\supset\Kset$ is a Liouvillian extension.
\item[(ii)]
If the differential Galois group ${\rm Gal}(\Kset/\Lset)$ is triangularizable,
 then $\Lset\supset\Kset$ is a strictly Liouvillian extension. 
\end{enumerate}
\end{theorem}

\subsection{Monodromy groups and Fuchsian equations}

Let $\Kset$ be the field of meromorphic functions on a Riemann surface $\Gamma$
 and let $z_0\in\Gamma$ be a nonsingular point in \eqref{LinearSystem}.
We prolong the fundamental matrix $\Phi(z)$ analytically
 along any loop $\ell$ based at $z_0$ and containing no singular point,
 and obtain another fundamental matrix $\ell*\Phi(z)$.
So there exists a constant nonsingular matrix $M_\ell$ such that
$$
\ell*\Phi(z) = \Phi(z)M_\ell.
$$
We call $M_\ell$ the \emph{monodromy matrix} for $\ell$.
The set of singularities in \eqref{LinearSystem}, which is denoted by $S$,
 is a discrete subset of $\Gamma$.
Let $\pi_1(\Gamma\setminus S,z_0)$
 be the fundamental group of homotopy classes of loops based at $z_0$.
The monodromy matrix $M_\ell$ depends
 on the homotopy class $[\ell]$ of the loop $\ell$,
 and it is also denoted by $M_{[\ell]}$.
We have a representation
$$
\mon\colon \pi_1(\Gamma\setminus S,z_0)\to {\rm GL}(n,\Cset), 
\quad [\ell]\mapsto M_{[\ell]}.
$$
The image of $\mon$ is called the \emph{monodromy group}
 of \eqref{LinearSystem}.
As in the differential Galois group,
 the representation $\mon$ depends on the choice of the fundamental matrix,
 but the monodromy group is defined as a group of matrices up to conjugation.
In general, monodromy transformations
 define automorphisms of the corresponding Picard-Vessiot extension.

Recall that equation~\eqref{LinearSystem} is said to be \emph{Fuchsian}
 if all singularities are regular.
For Fuchsian equations we have the following result
 (see, \eg, Theorem 5.8 in \cite{VS03} for the proof).
\begin{theorem}[Schlessinger]\label{th:sl}
Assume that equation~\eqref{LinearSystem} is Fuchsian.
Then the differential Galois group of \eqref{LinearSystem}
 is the Zariski closure of the monodromy group. 
\end{theorem}

Since the group of triangular matrices is algebraic,
 the Zariski closure of a triangularizable group is triangularizable.
Noting this fact,
 we obtain the following result immediately from Theorem~\ref{th:sl}.

\begin{corollary}\label{TriangularG}
Assume that equation~\eqref{LinearSystem} is Fuchsian.
Then the monodromy group is triangularizable
 if and only if the differential Galois group is triangularizable. 
\end{corollary}

\section{Proofs of Theorems~\ref{thm:main} and \ref{thm:main2}}

We first consider general second-order differential equations of the form 
\begin{equation}\label{SecondOrder}
\frac{\d^2 u}{\d z^2}+\frac{f_1(z)}{z}\frac{\d u}{\d z}+\frac{f_2(z)}{z^2}u = 0
\end{equation}
on $\Cset$, where $f_1(z)$ and $f_2(z)$ are holomorphic at $z=0$.
The origin $z=0$ is a regular singularity in \eqref{SecondOrder}.
Let $\rho,\rho'$ be the \emph{local exponents} of \eqref{SecondOrder} at $z=0$,
 \ie, roots of of the \emph{indicial equation}
\begin{equation}
s(s-1)+f_1(0)s+f_2(0)=0.
\label{eqn:ie}
\end{equation}
The following result is classical and well known (see, \eg, \cite{I56,WW27}).
\begin{lemma}
\label{lem:b}
Around $z=0$,
 equation~\eqref{SecondOrder} has two independent solutions of the following forms:
\begin{enumerate}
\item[(i)]
If $\rho-\rho'$ is not an integer, then
\[
u_1(z)=z^\rho v_1(z),\quad
u_2(z)=z^{\rho'}\!v_2(z);
\]
\item[(ii)]
if $\rho-\rho'$ is a nonnegative integer, then 
\[
u_1(z)=z^\rho v_1(z),\quad
u_2(z)=z^{\rho'}\!v_2(z)+u_1(z)\log z.
\]
\end{enumerate}
Here $v_1(z),v_2(z)$ denote some functions which are holomorphic at $z=0$.
\end{lemma}

Using Lemma~\ref{lem:b},
 we obtain the following result (cf. Lemma~4.6 of \cite{BY10}).

\begin{lemma}
\label{lem:c}
Suppose that the indicial equation~\eqref{eqn:ie}
 has roots $\rho_\pm$ such that $\Re(\rho_-)<0<\Re(\rho_+)$.
Then we have the following statements for \eqref{SecondOrder}:
\begin{enumerate}
\item[(i)]
There exists a nonzero solution $\bar{u}(z)$
 which is bounded along any ray approaching $z=0$.
\item[(ii)]
Any other independent solution is unbounded along any ray approaching $z=0$.
\item[(iii)]
Let $\ell$ be a loop around $z=0$ in $\Cset$.
The monodromy matrix $M_\ell$ has an eigenvalue $e^{2\pi i\rho_+}$
 and the bounded solution $\bar{u}(z)$ is the associated eigenvector.
\end{enumerate}
\end{lemma}

\begin{proof}
Parts~(i) and (ii) immediately follows from Lemma~\ref{lem:b}.
It remains to prove part~(iii).

Using Lemma~\ref{lem:b},
 we compute the monodromy matrix $M_\ell$ for the loop $\ell$ as
\[
M_\ell=
\begin{pmatrix}
e^{2\pi\rho_+} & 0\\
0 & e^{2\pi\rho_-}
\end{pmatrix}
\quad\mbox{and}\quad
\begin{pmatrix}
e^{2\pi\rho_+} & 0\\
2\pi i & e^{2\pi\rho_-}
\end{pmatrix}
\]
for $\rho_+-\rho_-\not\in\Zset$ and for $\rho_+-\rho_-\in\Zset$, respectively,
 in the basis $\{u_1(z),u_2(z)\}$.
Hence, $e^{2\pi\rho_+}$ is an eigenvalue of $M_\ell$
 and $u_1(z)$ is the associated eigenvector for both cases.
Noting that the bounded solution $\bar{u}(z)$ corresponds to $u_1(z)$,
 we prove part~(iii).
\end{proof}

Now we are in a position to prove Theorem~\ref{thm:main}.

\begin{proof}[Proof of Theorem~\ref{thm:main}]
Suppose that condition~\eqref{eqn:con} holds.
Then we have
\begin{equation}
\Re\left(\sqrt{\mu_\pm^2+4(\lambda-\nu_\pm)}\right)>|\mu_\pm|.
\label{eqn:conr}
\end{equation}
Here we took a branch of the square root function $\sqrt{z}$
 which is positive when $z$ is real and positive.
Noting that $f(z_\pm)=0$ and $f'(z_\pm)\neq 0$ by assumption~(A1),
 we write the indicial equations of \eqref{CE} at $z=z_\pm$ as
\begin{equation}
s(s-1)+(a_\pm\mu_\pm+1)s+a_\pm^2(\nu_\pm-\lambda)
 =s^2+a_\pm\mu_\pm s+a_\pm^2(\nu_\pm-\lambda)=0,
\label{eqn:ie2}
\end{equation}
where $a_\pm=1/f'(z_\pm)\neq 0$.
From \eqref{eqn:conr} we easily see that the indicial equation~\eqref{eqn:ie2}
 has roots with positive and negative real parts.
Hence, it follows from Lemmas~\ref{lem:b} and \ref{lem:c} that
 equation~\eqref{CE} has only one bounded independent solution of the form
\[
\psi_\pm(z)=(z-z_\pm)^{\chi_\pm}\,v_\pm(z)
\]
around each of $z=z_\pm$,
 where $\chi_\pm$ represent roots of \eqref{eqn:ie2} with positive real parts
 and $v_\pm(z)$ are holomorphic at $z=z_\pm$.
Moreover, by Lemma~\ref{lem:c}(iii),
 $\psi_\pm(z)$ are eigenvectors of the monodromy matrices $M_{\ell_\pm}$
 for loops $\ell_\pm$ around $z=z_\pm$ in $\Cset$.

Assume that the boundary value problem (\ref{SLTE},\ref{BC})
 has a nonzero solution $\psi(x)$.
Since the solution of \eqref{SLTE} must be represented
 as $\psi(x)=\psi_\pm(\gamma(x))$, we have $\psi_+(z)=\psi_-(z)$.
Hence, the monodromy matrices $M_{\ell_\pm}$ have a common eigenvector,
 so that the monodromy group for \eqref{CE} is triangularizable.
Appealing to Corollary~\ref{TriangularG}, we complete the proof.
\end{proof}

We turn to the proof of Theorem~\ref{thm:main2}.
Letting $\psi_1=\psi$ and $\psi_2=\d\psi/d x$,
 we rewrite \eqref{SLTE} as
\begin{equation}\label{ML}
\frac{\d}{\d x}
\begin{pmatrix}
\psi_1 \\
 \psi_2
\end{pmatrix}
=\begin{pmatrix}
0 & 1 \\ 
\lambda-\nu(x) & -\mu(x)
\end{pmatrix}\begin{pmatrix}
\psi_1 \\
\psi_2
\end{pmatrix}.
\end{equation}
The coefficient matrix of \eqref{ML} exponentially converges to
\[
A_\pm(\lambda)=
\begin{pmatrix}
0 & 1 \\ 
\lambda-\nu_\pm & -\mu_\pm
\end{pmatrix}
\]
as $x\to\pm\infty$.

\begin{lemma}
\label{lem:a}
Suppose that one of conditions \emph{(i)-(v)} in Theorem~\emph{\ref{thm:main2}} holds
 and the boundary value problem \emph{(\ref{SLTE},\ref{BC})} has a nonzero solution.
Then condition~\eqref{eqn:con} holds.
\end{lemma}

\begin{proof}
Let $\kappa_\pm$ be eigenvalues of the matrices $A_\pm(\lambda)$.
By a classical result on linear differential equations
 (see Section~8 and also Problems~29 and 35 in Chapter~3 of \cite{CL55}),
 we see that equation~\eqref{ML} has a nonzero solution $(\psi_1(x),\psi_2(x))$
 such that
\[
\lim_{x\to\pm\infty}\psi_j(x)e^{-\kappa_\pm x}=c_j,\quad
j=1,2
\]
for any constants $c_j$, $j=1,2$.

Assume that equation \eqref{ML} has a nonzero bounded solution.
Then for some eigenvalue $\kappa_\pm$,
 $e^{\kappa_\pm x}$ must tend to zero as $x\to\pm\infty$
 so that $\Re(\kappa+)<0$ and $\Re(\kappa_-)>0$.
This means that a root of the quadratic equation
\begin{equation}
s^2+\mu_\pm s-(\lambda-\nu_\pm)=0
\label{eqn:2eq}
\end{equation}
has negative and positive real parts for the signs $+$ and $-$, respectively.
Hence, conditions~(\ref{eqn:con}$_+$) and (\ref{eqn:con}$_-$)
 hold if $\mu_+\le 0$ and $\mu_-\ge 0$, respectively.
Here we have said that
 condition~(\ref{eqn:con}$_+$) and (\ref{eqn:con}$_-$) hold
 if condition~\eqref{eqn:con} holds for the signs ``$+$'' and ``$-$'', respectively. 
Moreover, if $\nu_+\ge\nu_-$ and $\nu_-\ge\nu_+$,
 then conditions~(\ref{eqn:con}$_+$) and (\ref{eqn:con}$_-$)
 means (\ref{eqn:con}$_-$) and (\ref{eqn:con}$_+$), respectively.
Thus, if one of (i)-(v) in Theorem~\ref{thm:main2} holds,
 then condition~\eqref{eqn:con} holds for both signs $\pm$.
\end{proof}

From the proof of Lemma~\ref{lem:a}
 we also see that eigenfunctions of the problem~(\ref{SLTE},\ref{BC})
 decay exponentially at infinity.

\begin{proof}[Proof of Theorem~\ref{thm:main2}]
Using Lemma~\ref{lem:a}, we obtain Theorem~\ref{thm:main2}
 as a corollary of \ref{thm:main}.
\end{proof}

\begin{remark}\label{rmk:a}

Suppose that equation~\eqref{CE} is Fuchsian on
 the Riemann sphere $\Pset^1$
 and has only three singularities at $z=z_{\pm}$ and $z_*$, where $z_*\in\Pset^1$.
Then the surface $\Gamma_\loc\setminus \{z_{\pm}\}$ is homotopic
 to $\Pset^1\setminus\{z_{\pm},z_*\}$,
so that they give rise to equivalent monodromy representations.
Hence, we see via Theorems~\ref{thm:main} and \ref{th:sl} that
equation~\eqref{CE} is integrable by Liouvillian functions on $\Cset(z)$,
which lie in a Liouvillian extension of $\Cset(z)$,
if the boundary value problem~(\ref{SLTE},\ref{BC}) has a nonzero solution.
This situation happens in examples of the next section.
\end{remark}

\begin{remark}\label{rmk:b}
As in the proof of Lemma~ \ref{lem:a}
 we show that if all eigenvalues of $A_+(\lambda)$ have negative real parts
 and an eigenvalue of $A_-(\lambda)$ has a positive real part,
 or if all eigenvalues of $A_-(\lambda)$ have positive real parts
 and an eigenvalue of $A_+(\lambda)$ has a negative real part,
 then the boundary value problem (\ref{SLTE},\ref{BC}) has a nonzero solution.
On the other hand, if all eigenvalues of $A_+(\lambda)$ have positive real parts
 or all eigenvalues of $A_-(\lambda)$ have negative real parts,
 then it has no nonzero solution.
Hence, $\lambda=\lambda_\R+i\lambda_\I$ is an eigenvalue
 of the problem (\ref{SLTE},\ref{BC}) if
\begin{enumerate}
\item[(i)]
$\mu_+>0$, $\mu_-\ge 0$, $\nu_-<\nu_+$ and
 condition~(\ref{eqn:con}$_-$) holds but
\begin{equation}
16\mu_+^2(\lambda_\R-\nu_+)+\lambda_I^2\le 0;
\label{eqn:con+}
\end{equation}
\item[(ii)]
$\mu_+\le 0$, $\mu_-<0$, $\nu_+<\nu_-$ and
 condition~(\ref{eqn:con}$_+$) holds but
\begin{equation}
16\mu_-^2(\lambda_\R-\nu_-)+\lambda_I^2\le 0;
\label{eqn:con-}
\end{equation}
\item[(iii)]
$\mu_-<0<\mu_+$ and
 both conditions~\eqref{eqn:con+} and \eqref{eqn:con-} hold.
\end{enumerate}
Note that these eigenvalues are continuous spectra
 for the eigenvalue problem (\ref{SLTE},\ref{BC}).
Moreover, $\lambda$ is not an eigenvalue of the problem (\ref{SLTE},\ref{BC}) if
\begin{enumerate}
\item[(i)]
$\mu_+\le 0$ and condition~\eqref{eqn:con+} holds;
\item[(ii)]
$\mu_-\ge 0$ and condition~\eqref{eqn:con-} holds.
\end{enumerate}
Thus, we can determine all eigenvalues using Theorem~\ref{thm:main}
 even if either of conditions~(i)-(v) in Theorem~\ref{thm:main2} do not hold.
\end{remark}

\section{Examples}

\begin{figure}[t]
\centering\includegraphics[scale=0.65]{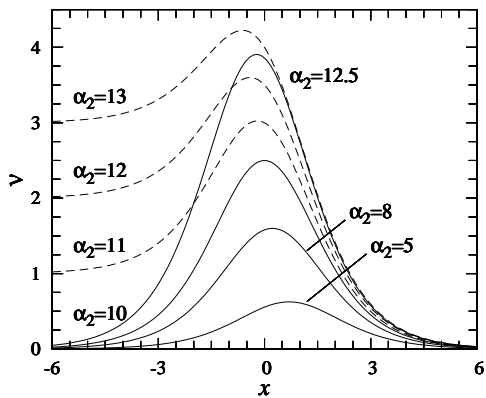}
\caption{Shape of the function $\nu(x)$ in \eqref{eqn:ex1}
 for several values of $\alpha_2$
 when $\alpha_1=10/\alpha_2$ or $1$ and $\alpha_3=10$.
Solid and dashed lines represent the cases of $\alpha_1=10/\alpha_2$ and $1$,
 respectively.
\label{fig:ex1V}}
\end{figure}

To illustrate the above theory, we give two examples
 with a Schr\"odinger equation having a generalized Hulth\'en potential
 and an eigenvalue problem for a traveling front in the Allen-Cahn equation.

\subsection{Schr\"odinger equation with a generalized Hulth\'en potential}
We first consider a case in which 
\begin{equation}
\mu(x)=0,\quad
\nu(x)=\frac{\alpha_2}{e^x+\alpha_1}-\frac{\alpha_3}{(e^x+\alpha_1)^2},
\label{eqn:ex1}
\end{equation}
where $\alpha_j$, $j=1,2,3$, are constants with $\alpha_j>0$, $j=1,2,3$.
For \eqref{eqn:ex1} equation~\eqref{SLTE} corresponds to a Schr\"odinger equation
 with the generalized Hulth\'en potential,
 which is a special case of \cite{Z81}.
A similar but more specific potential was also treated in \cite{K04,S02}.
We take $f(z)=z(1-z)$ so that equation~\eqref{eqn:1d} has two equilibria at $z=0,1$
 and a heteroclinic orbit
\[
\gamma(x)=\frac{e^x}{e^x+1}
\]
from $z=0$ to $z=1$.
We easily see that assumptions~(A1) and (A2) hold with $z_-=0$ and $z_+=1$,
\[
g(z)=0,\quad
h(z)=\frac{\alpha_2(z-1)}{(\alpha_1-1)z-\alpha_1}
 -\frac{\alpha_3(z-1)^2}{((\alpha_1-1)z-\alpha_1)^2}
\]
and condition~(i) in Theorem~\ref{thm:main2}, \ie, $\mu_\pm=0$, holds.
We also have
\[
\nu_-=\frac{\alpha_2}{\alpha_1}-\frac{\alpha_3}{\alpha_1^2},\quad
\nu_+=0,\quad
\sup_{x\in\Rset}\nu(x)=\frac{\alpha_2^2}{4\alpha_3}.
\]
See Figure~\ref{fig:ex1V} for the shape of the function $\nu(x)$
 with several values of $\alpha_2$ when $\alpha_1=10/\alpha_2$ or $1$ and $\alpha_3=10$.

Equation~\eqref{CE} becomes
\begin{equation}
\label{CEex1}
\psi''+\frac{2z-1}{z(z-1)}\psi'+\frac{h(z)-\lambda}{z^2(z-1)^2}\psi=0,
\end{equation}
which has only regular singularities at $z=0,1,z_0$, where
\[
z_0=
\begin{cases}
\alpha_1/(\alpha_1-1) & \mbox{for $\alpha_1\neq 1$};\\
\infty & \mbox{for $\alpha_1=1$}.
\end{cases}
\]
Solutions of \eqref{CEex1} are expressed by a Riemann $P$ function \cite{I56,WW27} as
\begin{equation}
P\left\{
\begin{matrix}
0 & 1 & z_0\\
\rho_1^+ & \rho_2^+ & \rho_3^+ & z\\
\rho_1^- & \rho_2^- & \rho_3^-
\end{matrix}
\right\},
\label{eqn:P}
\end{equation}
where $\rho_1^\pm$, $\rho_2^\pm$ and $\rho_3^\pm$
 represent the local exponents of \eqref{CEex1}
 at $z=0$, 1 and $z_0$, respectively, and are given by
\[
\rho_1^\pm=\pm\sqrt{\lambda-\nu_-},\quad
\rho_2^\pm=\pm\sqrt{\lambda},\quad
\rho_3^\pm=\half\left(1\pm\frac{1}{\alpha_1}\sqrt{\alpha_1^2+4\alpha_3}\right).
\]

The following result was essentially proved in \cite{K69}.

\begin{proposition}
\label{prop:kimura}
Consider a general Fuchsian second-order differential equation
 having three singularities $z=z_j$, $j=1,2,3$, and a Riemann $P$ function
\[
P\left\{
\begin{matrix}
z_1 & z_2 & z_3\\
\rho_1^+ & \rho_2^+ & \rho_3^+ & z\\
\rho_1^- & \rho_2^- & \rho_3^-
\end{matrix}
\right\}.
\]
Its monodromy and differential Galois groups are triangularizable
 if and only if at least one of $\rho_1+\rho_2+\rho_3$, $-\rho_1+\rho_2+\rho_3$,
 $\rho_1-\rho_2+\rho_3$ and $\rho_1+\rho_2-\rho_3$ is an odd integer,
 where $\rho_j=\rho_j^+ -\rho_j^-$, $j=1,2,3$, denote the exponent differences.
\end{proposition}

From Proposition~\ref{prop:kimura} we see that
 the monodromy and differential Galois groups for \eqref{CEex1} are triangularizable
 if and only if
\begin{equation}
\pm 2\sqrt{\lambda-\nu_-}\pm 2\sqrt{\lambda}\pm\bar{\rho}_3=2k+1
\label{eqn:CONex1a}
\end{equation}
for some combination of the signs, \ie,
\begin{equation}
\lambda=\frac{((2k+1\pm\bar{\rho}_3)^2+4\nu_-)^2}{16(2k+1\pm\bar{\rho}_3)^2}\in\Rset,
\label{eqn:CONex1b}
\end{equation}
where $k$ is some integer and $\bar{\rho}_3=\sqrt{\alpha_1^2+4\alpha_3}/\alpha_1$.

\begin{proposition}
\label{prop:a}
Real eigenvalues of the problem \emph{(\ref{SLTE},\ref{BC})}
 satisfy $\lambda\le\sup_{x\in\Rset}\nu(x)$.
\end{proposition}

\begin{proof}
Suppose that $\lambda>\sup_{x\in\Rset}\nu(x)$.
We can assume without loss of generality
 that a nontrivial solution of \eqref{ML} satisfies $\psi_2(x_0)>0$
 for some $x_0\in\Rset$
 since we can take $(-\psi_1(x_0),-\psi_2(x_0))$ if not.
If $\psi_1(x_0)>0$, then $\psi_1(x)$ does not converge to zero as $x\to+\infty$
 since $\psi_2^\prime(x)>0$ for $x>x_0$ when $\psi_2(x)$ is sufficiently small.
Similarly, if $\psi_1(x_0)<0$,
 then $\psi_1(x)$ does not converge to zero as $x\to-\infty$
 since $\psi_2^\prime(x)>0$ for $x<x_0$ when $\psi_2(x)$ is sufficiently small.
Thus, we obtain the result.
\end{proof}

Using Theorem~\ref{thm:main2}, Lemma~\ref{lem:a} and Proposition~\ref{prop:a},
 we prove the following.

\begin{theorem}\label{thm:ex1a}
If the boundary value problem \emph{(\ref{SLTE},\ref{BC})} with \eqref{eqn:ex1}
 has a nonzero solution, then condition~\eqref{eqn:CONex1b} holds
 and $\max(\nu_-,0)<\lambda<\alpha_2^2/4\alpha_3$.
\end{theorem}

Based on Theorem~\ref{thm:ex1a},
 we compute eigenvalues and eigenfunctions.
For $\alpha_1\neq 1$, we first transform \eqref{CEex1} by
\[
\zeta=\frac{(1-z_0)z}{z-z_0}\quad
\left(z=\frac{z_0\zeta}{\zeta+z_0-1}\right)
\]
to have regular singularities at $\zeta=0,1,\infty$.
We take $\zeta=z$ for $\alpha_1=1$.

Suppose that
\[
2\sqrt{\lambda-\nu_-}+2\sqrt{\lambda}=2k+1+\bar{\rho}_3>0,\quad
k\in\Zset.
\]
Then equation~\eqref{eqn:CONex1b} holds with the positive sign.
We set
\[
\eta(\zeta)=\zeta^{-\rho_1^+}(\zeta-1)^{-\rho_2^+}\psi(\zeta),
\]
so that the Riemann $P$ function \eqref{eqn:P} becomes
\[
\zeta^{-\rho_1^+}(\zeta-1)^{-\rho_2^+}
P\left\{
\begin{matrix}
0 & 1 & \infty\\
\rho_1^+ & \rho_2^+ & \rho_3^+ & \zeta\\
\rho_1^- & \rho_2^- & \rho_3^-
\end{matrix}
\right\}
=
P\left\{
\begin{matrix}
0 & 1 & \infty\\
0 & 0 & k+1+\bar{\rho}_3 & \zeta\\
2\rho_1^- & 2\rho_2^- & k+1
\end{matrix}
\right\}.
\]
Hence, we obtain the hypergeometric equation
\begin{equation}
\zeta(1-\zeta)\frac{\d^2\eta}{\d\zeta^2}
 +(c-(a+b+1)z)\frac{\d\eta}{\d\zeta}-ab\,\eta=0,
\label{eqn:hyperg}
\end{equation}
where $a=k+1+\bar{\rho}_3$, $b=k+1$ and $c=1-2\rho_1^-=1+2\sqrt{\lambda-\nu_-}$.
Thus, if $k$ is a negative integer,
 then there exists a bounded solution in \eqref{CEex1} as
\begin{equation}
\psi(\zeta)=\zeta^{\sqrt{\lambda-\nu_-}}(1-\zeta)^{\sqrt{\lambda}}
 F(k+1+\bar{\rho}_3,k+1,1+2\sqrt{\lambda-\nu_-};\zeta),
\label{eqn:solex1}
\end{equation}
where $F(a,b,c;\zeta)$ is the hypergeometric function
\[
F(a,b,c;\zeta)=\sum_{j=0}^{\infty}\frac{a(a+1)\cdots(a+j-1)b(b+1)\cdots(b+j-1)}
 {j!\,c(c+1)\cdots(c+j-1)}\zeta^j,
\] 
which becomes a finite series when $a$ or $b$ is a nonpositive integer.
For the other cases of \eqref{eqn:CONex1a},
 similar computations show that there is no bounded solution in \eqref{CEex1}.
Thus, we have the following result.

\begin{theorem}\label{thm:ex1b}
If for some integer $k\in(-\half(\bar{\rho}_3+1),0)$
\[
\lambda=\frac{((2k+1+\bar{\rho}_3)^2+4\nu_-)^2}{16(2k+1+\bar{\rho}_3)^2}
 \in\left(\max(\nu_-,0),\frac{\alpha_2^2}{4\alpha_3}\right),
\]
then the boundary value problem \emph{(\ref{SLTE},\ref{BC})} with \eqref{eqn:ex1}
 has a nonzero solution given by \eqref{eqn:solex1}
 with $\zeta=(1-z_0)\gamma(x)/(\gamma(x)-z_0)$ for $\alpha_1\neq 1$
 and $\zeta=\gamma(x)$ for $\alpha_1=1$.
\end{theorem}

\begin{figure}[t]
\centering\includegraphics[scale=0.65]{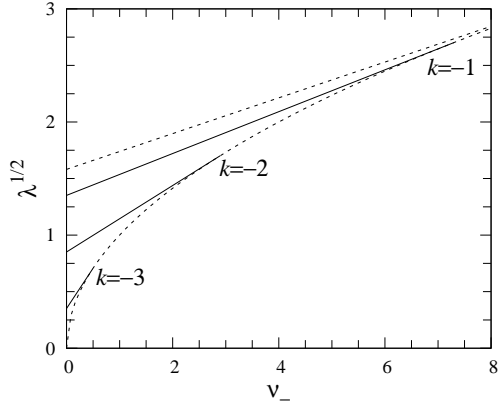}
\caption{Eigenvalues for \eqref{eqn:ex1} with $\alpha_1=1$ and $\alpha_3=10$.
The dotted lines represent the upper bound
 $\sqrt{\lambda}=\half\alpha_2/\sqrt{\alpha_3}
 =\half(\alpha_1\nu_-/\sqrt{\alpha_3}+\sqrt{\alpha_3}/\alpha_1)$
 and the lower bound $\sqrt{\lambda}=\sqrt{\nu_-}$.
\label{fig:ex1ev}}
\end{figure}

\begin{figure}[t]
\includegraphics[scale=0.47]{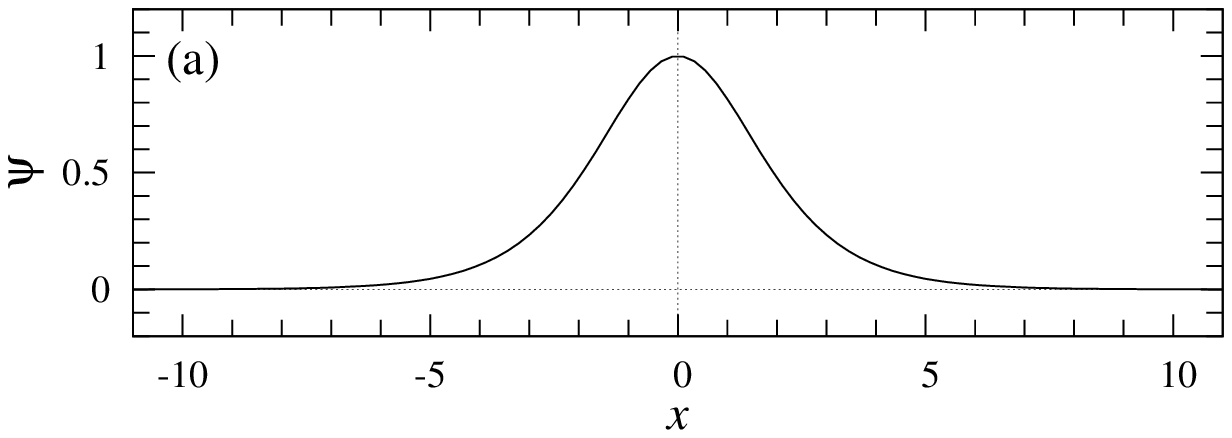}\quad
\includegraphics[scale=0.47]{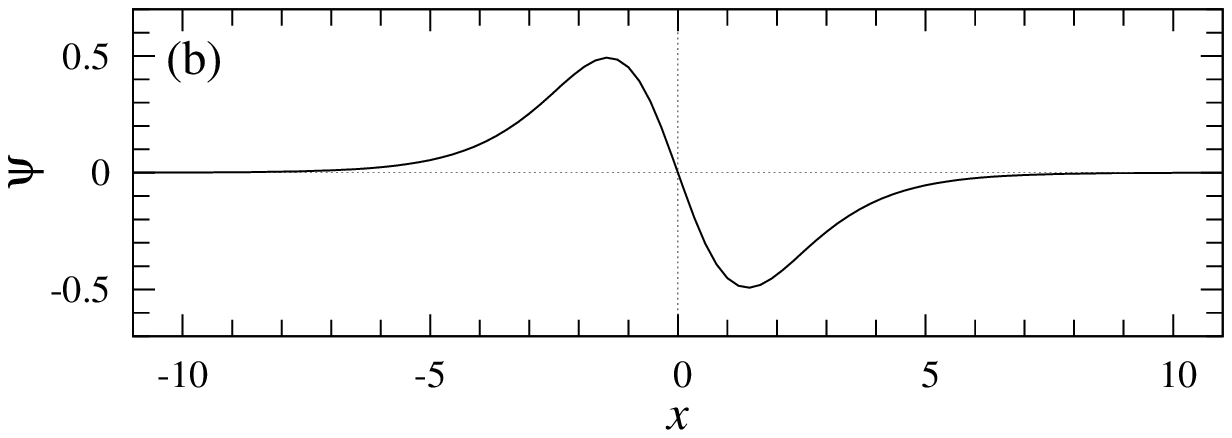}\\[2ex]
\includegraphics[scale=0.47]{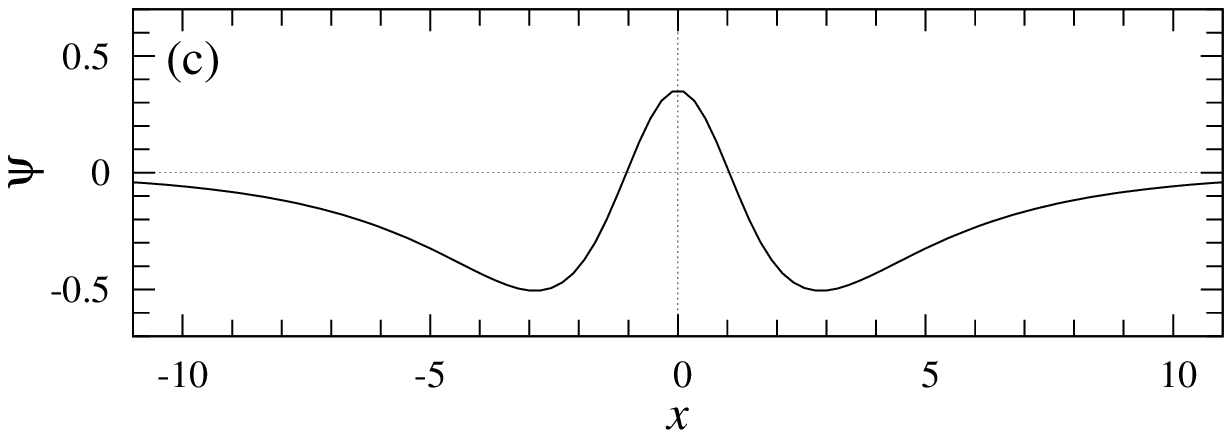}\quad
\includegraphics[scale=0.47]{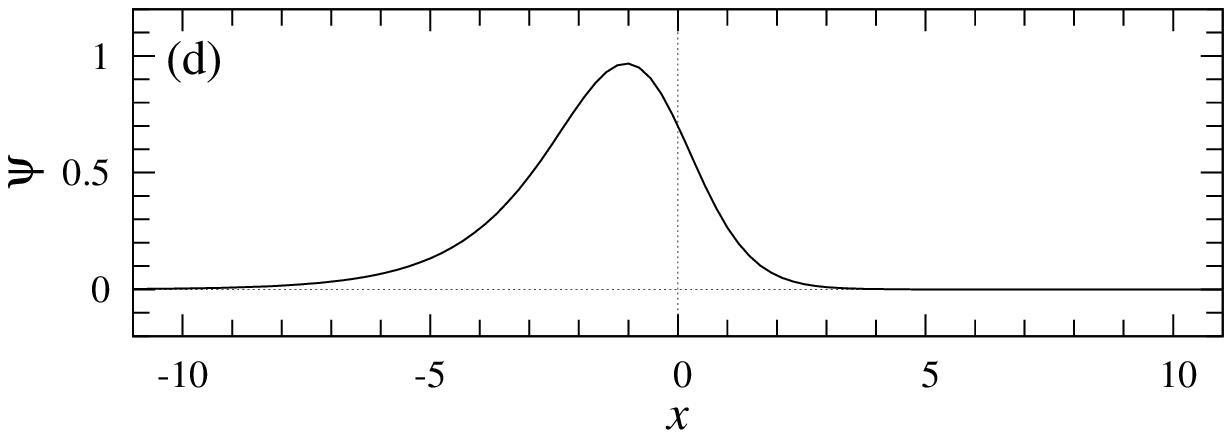}\\[2ex]
\includegraphics[scale=0.47]{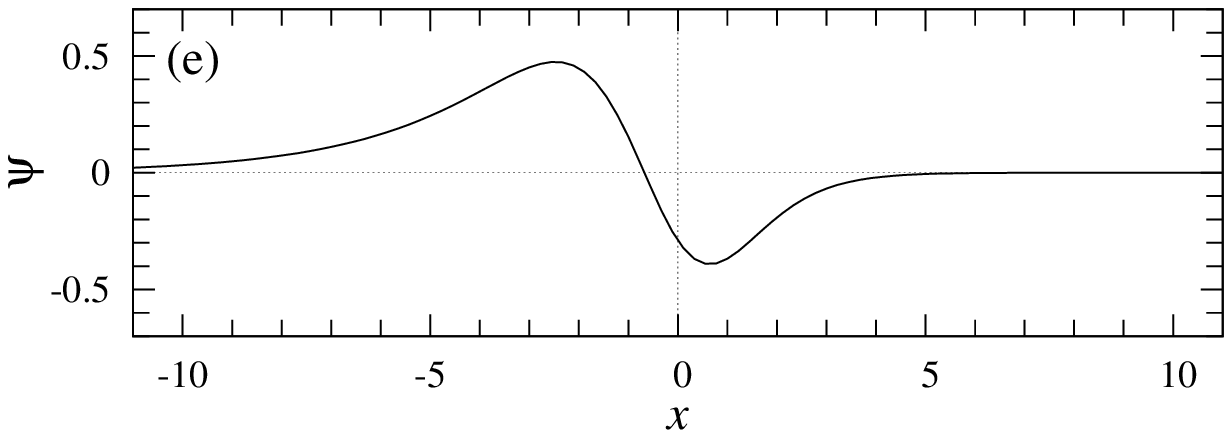}\quad
\includegraphics[scale=0.47]{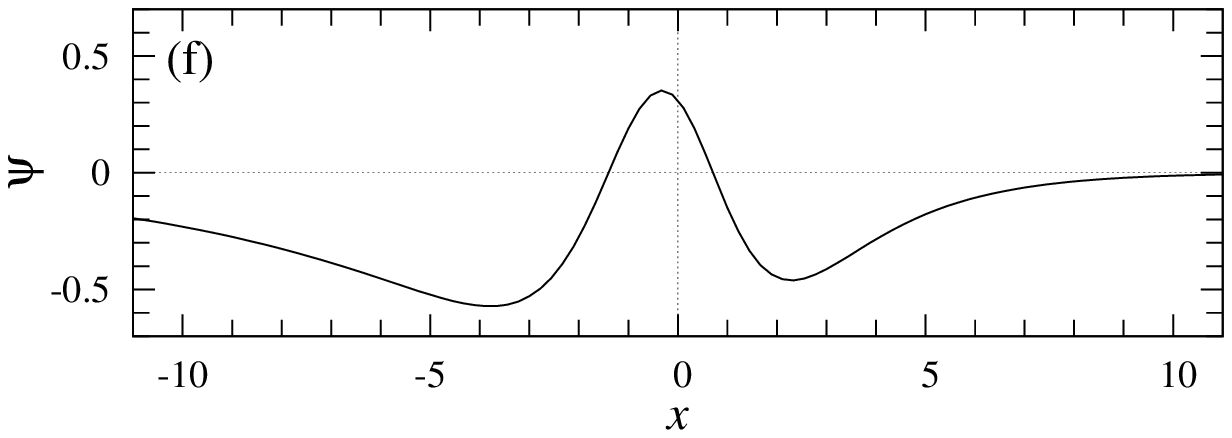}
\caption{Eigenfunctions for \eqref{eqn:ex1} with $\alpha_1=1$ and $\alpha_3=10$:
(a) $(\nu_-,\sqrt{\lambda})=(0,1.35078)$;
(b) $(0,0.850781)$;
(c) $(0,0.350781)$;
(d) $(3.5,1.99855)$;
(e) $(1.5,1.29155)$;
(f) $(0.25,0.528955)$.
\label{fig:ex1}}
\end{figure}

Eigenvalues and eigenfunctions
 for \eqref{eqn:ex1} with $\alpha_1=1$ and $\alpha_3=10$ are plotted
 in Figs.~\ref{fig:ex1ev} and \ref{fig:ex1}, respectively.
Eigenfunctions on the first, second and third branches in Fig.~\ref{fig:ex1ev}
 are given in Figs.~\ref{fig:ex1}(a,d), (b,e) and (c,f), respectively.
Note that the hypothesis of Theorem~\ref{thm:ex1b} holds only for $k=-1,-2,-3$
 since $\bar{\rho}_3=\sqrt{\alpha_1^2+4\alpha_3}/\alpha_1=\sqrt{41}=6.4\ldots$.
 
\subsection{Spectral stability of a front in the Allen-Cahn equation}

\begin{figure}[t]
\centering\includegraphics[scale=0.65]{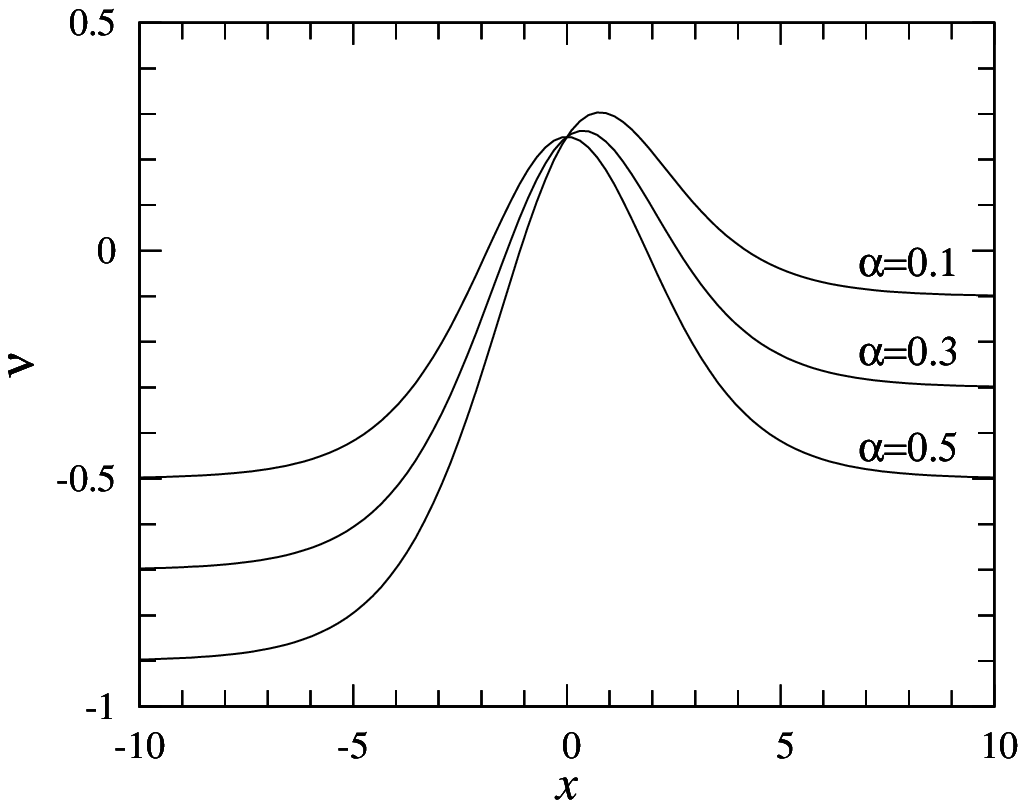}
\caption{Shape of the function $\nu(x)$ in \eqref{eqn:ex2} for $\alpha=0.1,0.3,0.5$.
Note that the corresponding functions for $\alpha$ and $1-\alpha$
 are symmetric about $x=0$.
\label{fig:ex2V}}
\end{figure}

\begin{figure}[t]
\centering\includegraphics[scale=0.57]{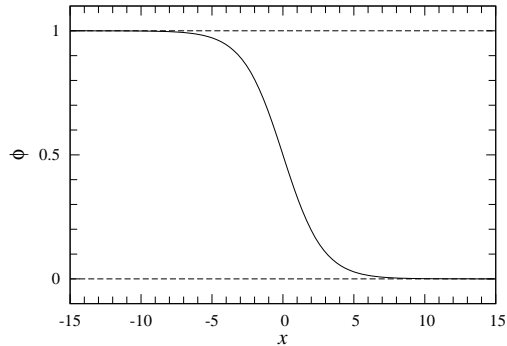}
\caption{Front solution \eqref{eqn:front}
 in the Allen-Cahn equation~\eqref{eqn:AC}.
\label{fig:ex2front}}
\end{figure}

We next consider a case in which
\begin{equation}
\mu(x)=\sqrt{2}(\half-\alpha),\quad
\nu(x)=-3\phi^2(x)+2(\alpha+1)\phi(x)-\alpha,
\label{eqn:ex2}
\end{equation}
where $\alpha$ is a constant
 such that $0<\alpha<1$ and
\begin{equation}
\phi(x)=\frac{1}{e^{x/\sqrt{2}}+1}.
\label{eqn:front}
\end{equation}
For \eqref{eqn:ex2} the eigenvalue problem~(\ref{SLTE},\ref{BC})
 is related to spectral stability
 of a traveling front solution with the velocity $c=\sqrt{2}(\half-\alpha)$,
\[
u(t,x)=\phi(x-ct),
\]
in a PDE called the Allen-Cahn (or Nagumo) equation
\begin{equation}
\frac{\partial u}{\partial t}
 =\frac{\partial^2 u}{\partial x^2}+u(1-u)(u-\alpha).
\label{eqn:AC}
\end{equation}
Asymptotic stability of traveling front solutions in such PDEs
 was studied in \cite{FM77,FM81,C97}
 without solving the associated eigenvalue problem.
Essentially the same eigenvalue problem as \eqref{eqn:ex2} 
 was also considered in \cite{S76,S77}.

We take $f(z)=z(1-z)/\sqrt{2}$
 so that equation~\eqref{eqn:1d} also has a heteroclinic orbit
\[
\gamma(x)=\frac{e^{x/\!\sqrt{2}}}{e^{x/\!\sqrt{2}}+1}
\]
from $z=0$ to $z=1$.
We easily see that assumptions~(A1) and (A2) hold with $z_-=0$, $z_+=1$ and
\[
g(z)=\sqrt{2}(\half-\alpha),\quad
h(z)=-3z^2+2(2-\alpha)z+\alpha-1.
\]
Condition~(i) in Theorem~\ref{thm:main2} holds for $\alpha=\half$
 but conditions~(i)-(v) do not hold for $\alpha\neq\half$ since
\[
\mu_\pm=\sqrt{2}(\half-\alpha),\quad
\nu_-=\alpha-1,\quad
\nu_+=-\alpha.
\]
so that $\mu_\pm>0$ and $\nu_+>\nu_-$ for $\alpha\in(0,\half)$,
 and $\mu_\pm<0$ and $\nu_->\nu_+$ for $\alpha\in(\half,1)$.
We also have
\[
\sup_{x\in\Rset}\nu(x)=\third(\alpha^2-\alpha+1)>0.
\]
See Figs.~\ref{fig:ex2V} and \ref{fig:ex2front}
 for the shapes of the function $\nu(x)$ with $\alpha=0.1,0.3,0.5$
 and the front solution $\phi(x)$.

Equation~\eqref{CE} becomes
\begin{equation}
\label{CEex2}
\psi''+\frac{2(z+\alpha-1)}{z(z-1)}\psi'+\frac{2(h(z)-\lambda)}{z^2(z-1)^2}\psi=0,
\end{equation}
which has only regular singularities at $z=0,1,\infty$.
Solutions of \eqref{CEex2} are expressed by a Riemann $P$ function
 as \eqref{eqn:P} with $z_0=\infty$ and
\begin{align*}
&
\rho_1^\pm=\half(2\alpha-1\pm\sqrt{8\lambda+(2\alpha-3)^2}),\\
&
\rho_2^\pm=\half(1-2\alpha\pm\sqrt{8\lambda+(2\alpha+1)^2}),\quad
\rho_3^+=3,\quad
\rho_3^-=-2.
\end{align*}
Using Proposition~\ref{prop:kimura}, we see that
 the monodromy and differential Galois groups for \eqref{CEex2} are triangularizable
 if and only if
\begin{equation}
\pm\sqrt{8\lambda+(2\alpha-3)^2}\pm\sqrt{8\lambda+(2\alpha+1)^2}=2k
\label{eqn:CONex2a}
\end{equation}
for some combination of the signs, \ie,
\begin{equation}
\lambda=\frac{(k^2-4)(k+1-2\alpha)(k-1+2\alpha)}{8k^2}\in\Rset,
\label{eqn:CONex2b}
\end{equation}
where $k$ is some integer.
By Theorems~\ref{thm:main} and \ref{thm:main2}, Remark~\ref{rmk:b}
 and Proposition~\ref{prop:a}, we obtain the following result.

\begin{theorem}\label{thm:ex2a}
If the boundary value problem \emph{(\ref{SLTE},\ref{BC})} with \eqref{eqn:ex2}
 has a nonzero solution, then one of the following conditions holds:
\begin{enumerate}
\item[(i)]
Condition~\eqref{eqn:CONex2b} holds
 with $\max(\alpha-1,-\alpha)<\lambda<\third(\alpha^2-\alpha+1)$;
\item[(ii)]
$\alpha\in(0,\half)$ and conditions~\emph{(\ref{eqn:con}$_-$)}
 and \eqref{eqn:con+} hold;\\[-2ex]
\item[(iii)]
$\alpha\in(\half,1)$ and conditions~\emph{(\ref{eqn:con}$_+$)}
 and \eqref{eqn:con-} hold.
\end{enumerate}
Moreover, if condition~\emph{(ii)} or \emph{(iii)} holds,
 then the boundary value problem \emph{(\ref{SLTE},\ref{BC})} with \eqref{eqn:ex2}
 has a nonzero solution.
\end{theorem}

\begin{figure}[t]
\centering\includegraphics[scale=0.55]{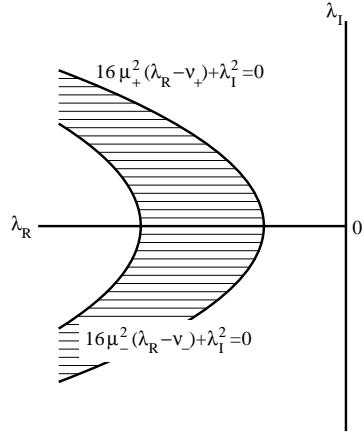}
\caption{Continuous spectra (the shaded region) for \eqref{eqn:ex2}
 when $\alpha\in(0,\half)$.
Note that $\nu_+>\nu_-$ in this case.
\label{fig:ex2cev}}
\end{figure}

See Fig.~\ref{fig:ex2cev} for continuous spectra
 detected in Theorem~\ref{thm:ex2a}(ii) for $\alpha\in(0,\half)$.
A similar picture can be drawn for $\alpha\in(\half,1)$.
Such continuous spectra in Theorem~\ref{thm:ex2a}(ii) and (iii)
 were briefly given in \cite{S76,S77}.

Based on Theorem~\ref{thm:ex2a}(i),
 we compute eigenvalues and eigenfunctions.
Suppose that
\[
\sqrt{8\lambda+(2\alpha-3)^2}+\sqrt{8\lambda+(2\alpha+1)^2}=2k>0,\quad
k\in\Zset.
\]
We set
\[
\eta(z)=z^{-\rho_1^+}(z-1)^{-\rho_2^+}\psi(z),
\]
so that the Riemann $P$ function \eqref{eqn:P} becomes
\begin{align*}
&
z^{-\rho_1^+}(z-1)^{-\rho_2^+}
P\left\{
\begin{matrix}
0 & 1 & \infty\\
\rho_1^+ & \rho_2^+ & 3 & z\\
\rho_1^- & \rho_2^- & -2
\end{matrix}
\right\}\\
&=
P\left\{
\begin{matrix}
0 & 1 & \infty\\
0 & 0 & k+3 & z\\
-\sqrt{8\lambda+(2\alpha-3)^2} & -\sqrt{8\lambda+(2\alpha+1)^2} & k-2
\end{matrix}
\right\}.
\end{align*}
Hence, we obtain the hypergeometric equation \eqref{eqn:hyperg}
 with $a=k+3$, $b=k-2$ and $c=1+\sqrt{8\lambda+(2\alpha-3)^2}$.
Thus, if $k=1,2$, then there exists a bounded solution in \eqref{CEex2} as
\begin{equation}
\psi(z)=z^{\rho_1^+}(1-z)^{\rho_2^+}
 F(k+3,k-2,1+\sqrt{8\lambda+(2\alpha-3)^2};z).
\label{eqn:solex2}
\end{equation}
For the other cases of \eqref{eqn:CONex2a},
 similar computations show that there is no bounded solution in \eqref{CEex2}.
Noting that equation~\eqref{eqn:CONex2b} is not positive for $k=1,2$,
 we prove the following result.

\begin{theorem}\label{thm:ex2b}
If $\lambda=0$ and
\[
\lambda={\textstyle\frac{3}{2}}\alpha(\alpha-1),\quad
\alpha\in(\third,{\textstyle\frac{2}{3}}),
\]
respectively,
 then the boundary value problem \emph{(\ref{SLTE},\ref{BC})} with \eqref{eqn:ex2}
 has nonzero solutions given by
\begin{equation}
\psi(x)=\frac{e^{x/\!\sqrt{2}}}{(e^{x/\!\sqrt{2}}+1)^2}
\label{eqn:solex2a}
\end{equation}
and
\begin{equation}
\psi(x)=\frac{e^{(1-\alpha)x/\!\sqrt{2}}}{e^{x/\!\sqrt{2}}+1}
 \left(1-\frac{1}{1-\alpha}\frac{e^{x/\!\sqrt{2}}}{e^{x/\!\sqrt{2}}+1}\right).
\label{eqn:solex2b}
\end{equation}
\end{theorem}

\begin{proof}
When $k=1$, we have
\[
\lambda={\textstyle\frac{3}{2}}\alpha(\alpha-1),
\]
by \eqref{eqn:CONex2b} and $\alpha\in(\third,{\textstyle\frac{2}{3}})$
 since $\lambda>\max(\alpha-1,-\alpha)$.
Hence, we obtain
\[
\rho_1^+=1-\alpha,\quad
\rho_2^+=\alpha,\quad
c=4-4\alpha,
\]
and write \eqref{eqn:solex2} as
\[
\psi(z)=z^{1-\alpha}(1-z)^\alpha\left(1-\frac{z}{1-\alpha}\right),
\]
which yields \eqref{eqn:solex2b} by $z=\gamma(x)$.
On the other hand, when $k=2$,
 we have $\lambda=0$ and $\rho_1^+=\rho_2^+=1$,
 so that equation~\eqref{eqn:solex2} becomes
\[
\psi(z)=z(1-z),
\]
which yields \eqref{eqn:solex2a}.
\end{proof}

\begin{remark}
The eigenfunction~\eqref{eqn:solex2a} for $\lambda=0$ can be written as
\[
\psi(x)=-\sqrt{2}\,\frac{\d\phi}{\d x}(x).
\]
The existence of this eigenfunction is also guaranteed
 by the invariance of the PDE \eqref{eqn:AC} under the group of translations
 $x\mapsto x+x_0$, $x_0\in\Rset$.
\end{remark}

\begin{figure}[t]
\centering\includegraphics[scale=0.65]{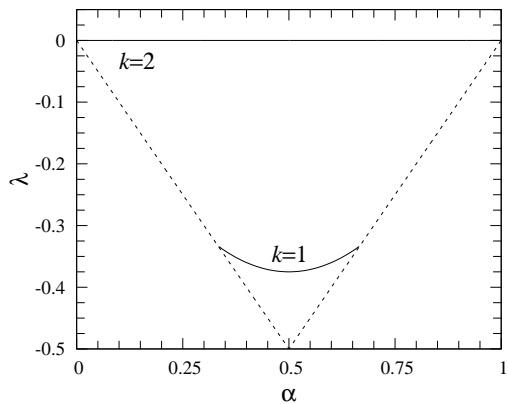}
\caption{Eigenvalues given by \eqref{eqn:CONex2b} for $k=1,2$
 in the case of \eqref{eqn:ex2}.
The dotted line represents the lower bound $\lambda=\max(\alpha-1,-\alpha)$.
\label{fig:ex2ev}}
\end{figure}

\begin{figure}[t]
\includegraphics[scale=0.47]{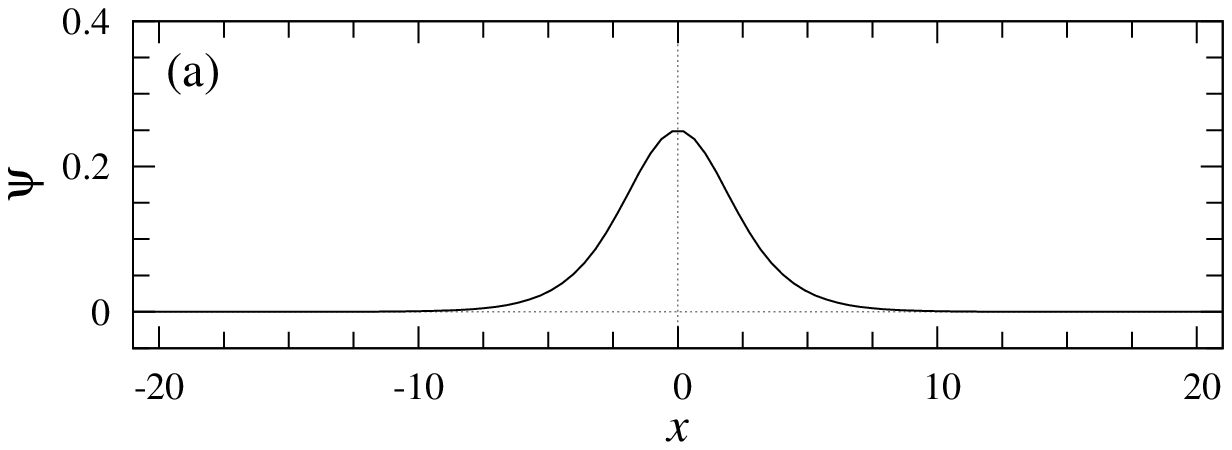}\quad
\includegraphics[scale=0.47]{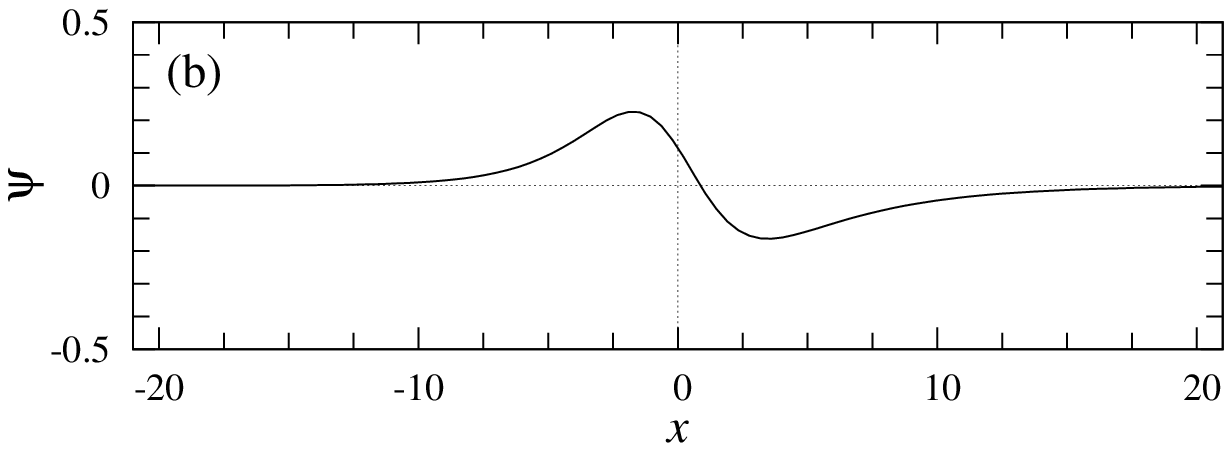}\\[2ex]
\includegraphics[scale=0.47]{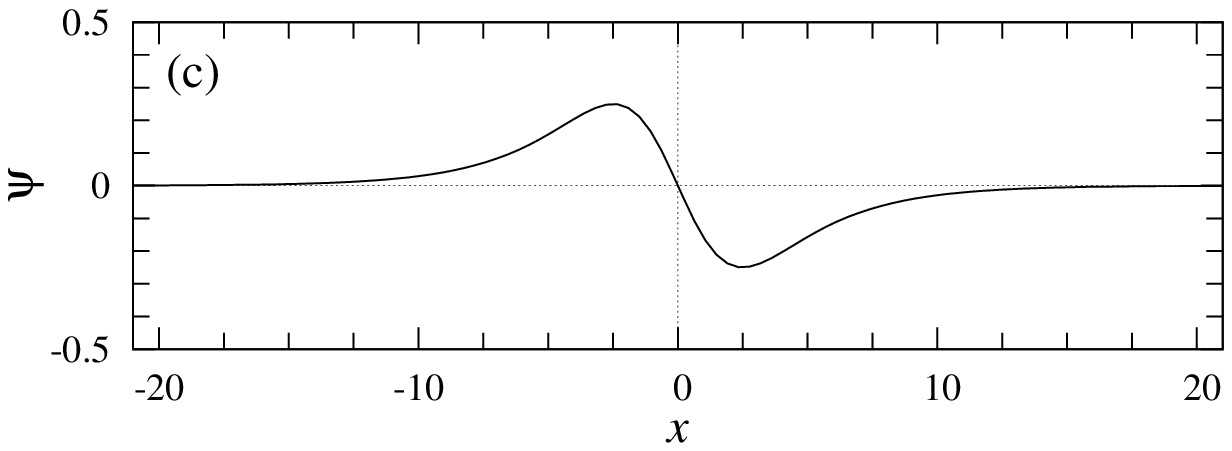}\quad
\includegraphics[scale=0.47]{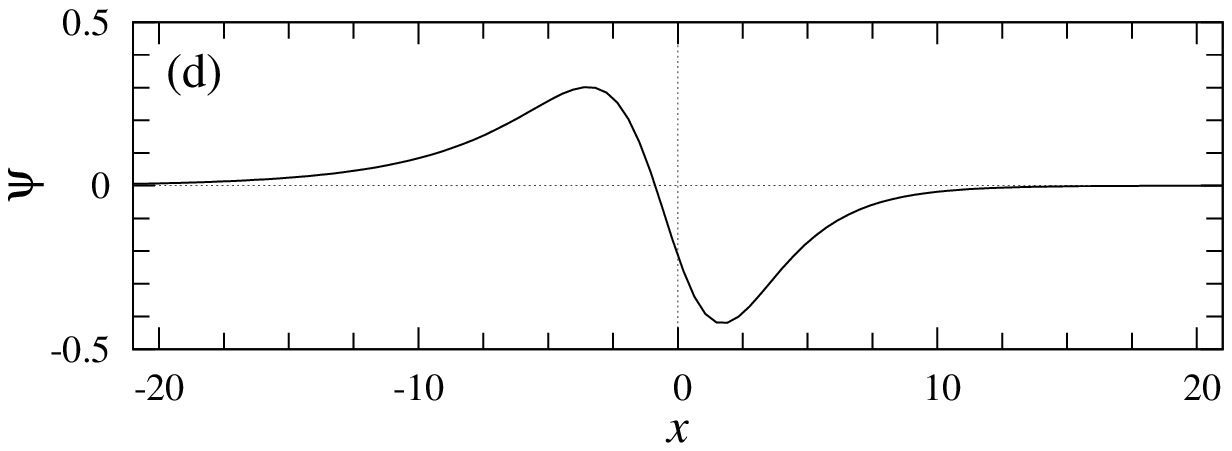}
\caption{Eigenfunctions for \eqref{eqn:ex2}:
(a) $\lambda=0$;
(b) $\alpha=0.35$;
(c) $0.5$;
(d) $0.65$.
Plates~(b)-(c) show the functions for $\lambda=\frac{3}{2}\alpha(\alpha-1)$.
\label{fig:ex2}}
\end{figure}

Eigenvalues and eigenfunctions for \eqref{eqn:ex2} are plotted
 in Figs.~\ref{fig:ex2ev} and \ref{fig:ex2}, respectively.
Note that there exist continuous real spectra
 between $\max(\alpha-1,-\alpha)$ and $\min(\alpha-1,-\alpha)$.
Nonzero discrete eigenvalues and the associated eigenfunctions
 in this eigenvalue problem were not given in \cite{S76,S77}.


\noindent{\sc Kazuyuki~Yagasaki \\
Mathematics Division, \\
Department of Information Engineering, \\
Niigata University,\\
Niigata 950-2181, Japan} \\
{\tt yagasaki@ie.niigata-u.ac.jp}

\medskip

\noindent{\sc David~Bl\'azquez-Sanz \\
Mathematics Division, \\
Department of Information Engineering, \\
Niigata University,\\
Niigata 950-2181, Japan} \\
\smallskip
Permanent adress: \\
{\sc Escuela de Matem\'aticas,\\
Sergio Arboleda University,\\
Bogot\'a, Colombia \\
}
{\tt david@ima.usergioarboleda.edu.co}

\end{document}